\documentclass[11pt, reqno]{amsart}
\usepackage{amsmath, amsthm, amscd, amsfonts, amssymb, latexsym, graphicx, color, stmaryrd}
\usepackage{MnSymbol}
\usepackage{mathrsfs}
\usepackage[all,cmtip]{xy}	% diagrams
\usepackage[bookmarksnumbered, colorlinks, plainpages]{hyperref}
\usepackage{enumitem}
\usepackage{bbm}
\usepackage{slashed}
\usepackage{tikz-cd}

\tikzset{%
    symbol/.style={%
        ,draw=none
        ,every to/.append style={%
            edge node={node [sloped, allow upside down, auto=false]{$#1$}}}
    }
}

\def\presuper#1#2%
  {\mathop{}%
   \mathopen{\vphantom{#2}}^{#1}%
   \kern-\scriptspace%
   #2}

\newcommand{\Nn}{\mathbb{N}}
\newcommand{\Rr}{\mathbb{R}}
\newcommand{\Cc}{\mathbb{C}}

\newcommand{\Mor}{\mathrm{Mor}}

\newcommand{\Diff}{\mathrm{Diff}}	% cat. diffeological spaces
\newcommand{\Man}{\mathrm{Man}}	% cat. smooth (fin. dim.) manifolds
\newcommand{\Manec}{\mathrm{Man}_{\mathrm{ec}}}	% cat. mf with emb. corners
\newcommand{\ManC}{\mathrm{Man}_{\Cc}}		% cat. of complex manifolds
\newcommand{\ManRA}{\mathrm{Man}_{\mathrm{RA}}}	% cat. of real analytic manifolds
\newcommand{\ManAlg}{\mathrm{Man}_{\mathrm{Alg}}}	% cat. of algebraic manifolds
\newcommand{\ManFr}{\mathrm{Man}_{\mathrm{Fr}}}	%  cat. of Frechet manifolds
\newcommand{\Top}{\mathrm{Top}}	% cat. top. spaces
\newcommand{\Ee}{\mathbb{E}} 	% adm. epimorphisms
\newcommand{\Eeloc}{\mathbb{E}_{\mathrm{loc}}}
\newcommand{\Aa}{\mathbb{A}}	% ambient category
\newcommand{\AG}{\Aa \G}	% cat. of internal groupoids
\newcommand{\AGbi}{\AG_{\mathrm{bi}}}	% bicat. of internal groupoids

\newcommand{\AS}{\Aa \mathcal{S}}	% cat. of internal spans
\newcommand{\Tau}{\mathcal{T}}		% Groth. top. 
\newcommand{\TauDopen}{\mathcal{T}_{\mathrm{D-open}}}
\newcommand{\Cech}{\mathrm{\check{C}}}	% Cech-groupoid

\newcommand{\Sh}{\mathrm{Sh}}
\newcommand{\PhiPsi}{\mathrm{\presuper{\Phi}{\Psi}}}	% diffeological pseudo-group
\newcommand{\id}{\mathrm{id}}

\renewcommand{\H}{\mathcal{H}} % Hilbert space

\newcommand{\G}{\mathcal{G}}

\newcommand{\F}{\mathcal{F}}
\renewcommand{\L}{\mathcal{L}} % bd operators

\newcommand{\X}{\mathcal{X}}	% blow-up
	% ...

\renewcommand{\P}{\mathcal{P}}		% projs
\renewcommand{\O}{\mathcal{O}}

\newcommand{\U}{\mathcal{U}}

\newcommand{\Bun}[3]{#1 \ \rotatebox[origin=c]{90}{$\circlearrowleft$}\ #2 \ \rotatebox[origin=c]{90}{$\circlearrowright$}\ #3}
\newcommand{\pr}{\operatorname{pr}}

\newcommand{\iso}{\xrightarrow{\sim}}	% iso arrow

\newtheorem{Thm}{Theorem}[section]
\newtheorem{Lem}[Thm]{Lemma}
\newtheorem{Prop}[Thm]{Proposition}

\theoremstyle{definition}
\newtheorem{Def}[Thm]{Definition}

\newtheorem{Rem}[Thm]{Remark}

\newtheorem{Not}[Thm]{Notation}

\begin{document}
\setcounter{page}{1}

%-------------------------- Pleased do not change the following line-------------------------------------------
%\noindent \textcolor[rgb]{0.99,0.00,0.00}{This is a submission to one of journals of TMRG: BJMA/AFA}\\[.5in]
%--------------------------------------------------------------------------------------------------------------

\title[Internal Absolute Geometry]{Internal Absolute Geometry I: Desingularization}

\author[Karsten Bohlen]{Karsten Bohlen}

%\address{$^{1}$ Princeton Research Forum}
%\email{\textcolor[rgb]{0.00,0.00,0.84}{kbohlen@gmail.com}}

%\dedicatory{This paper is dedicated to Professor ABCD}



% \date{Received: xxxxxx; Revised: yyyyyy; Accepted: zzzzzz.}

%\maketitle

\begin{abstract}
We introduce an axiomatization of Grothendieck sites with additional structure, and we describe sheaves that reconstruct groupoids which are internal to the site structure. This setting applies to various concrete situations, where a Nash blowup of a singular space results in an almost regular foliation. It also potentially applies to various types of moduli spaces. The sheaf can encode candidate holonomy groupoids that desingularize such spaces. 
\end{abstract} 

\maketitle

\section{Introduction}

% Intro
A singular foliation, i.e. a foliation with leaves of varying dimension, over a smooth manifold, possesses an associated holonomy groupoid. The construction furnishes a groupoid with smooth source fibers, the orbits of which align precisely with the leaves of the foliation. C. Ehresmann initially studied the holonomy groupoid \cite{ehresmann}. It was later extended to the regular foliation case by H. E. Winkelnkemper \cite{winkelnkemper} and to singular foliations by J. Pradines \cite{pradines}. More recent contributions have come from C. Debord \cite{debord}, G. Skandalis, and I. Androulidakis \cite{as2}. Notably, a singular foliation, as defined by Androulidakis-Skandalis, becomes a Debord foliation after a Nash blowup \cite{l}.

Various authors have expanded this concept to other settings. For instance, the holonomy groupoid has been constructed for Teichm\"uller and Riemann moduli spaces when viewed as stacks. This is connected with partially foliated structures based on infinite-dimensional Fr\'echet manifolds \cite{meersseman}. Generally, even if one's focus isn't strictly on smooth systems, the holonomy groupoid can offer a valuable geometric interpretation of structures similar to foliation. For instance, within cognitive systems, the geometry of information spaces is smooth only in a generalized (diffeological) sense \cite{gw}. The setting of this work is an ambient category with additional structure and groupoids internal to this category. 
% Cat. of virtual manifolds - special kinds of sheaves (of "holonomy type").
% Singular foliations.
% ...
% Backwards problem: Instead of considering a given singular structure and describing an atlas
% construction to desingularize the structure, we start with giving an atlas structure, in abstract terms, which may then serve as a blueprint for desingularization constructions of various types.
Some notable instances of specific ambient site structures $(\Aa, \Ee, \Tau)$ are as follows:

\begin{enumerate}
\item The case of smooth manifolds $\Aa = \Man$, or compact smooth manifolds with (embedded) corners, $\Aa = \Manec$, with $\Tau$ the open covers and $\Ee$ the (tame) surjective submersions. 

\item Complex manifolds, where the ambient category is $\Aa = \ManC$, $\Tau$ are the open covers, $\Ee$ are surjective holomorphic submersions. 

\item Real analytic manifolds, $\Aa = \ManRA$, with $\Tau$ the open covers, $\Ee$ the surjective real analytic submersions.

\item Algebraic manifolds, $\Aa = \ManAlg$, $\Tau$ the \'etale covers, $\Ee$ the surjective \'etale morphisms. 
\end{enumerate}

Other examples may concern moduli spaces, e.g. atlas constructions for orbifold stratified spaces that are modeled by \'etale Lie groupoids, \cite{chenetal}. The category $\Aa = \ManFr$ of Fr\'echet manifolds fulfills the assumptions \ref{item:site-structure} through \ref{item:principal-bundle} that are stated below, cf. \cite{mz}[Section 9], and the embedding into the category of diffeological spaces is fully faithful, cf. \cite{losik, wockel}. Let us recall the discussion of a singular foliation from \cite{l}: A subsheaf $\F$ of the sheaf of vector fields $\X$ that is closed with regard to Lie bracket and locally finitely generated as an $\O$-module for the relevant sheaf of functions $\O$. After a single Nash blowup, this results in a Debord foliation: the data $(\widetilde{M}, \pi)$ with $\pi \colon \widetilde{M} \to M$ is onto and proper; $\pi_{M_{\mathrm{reg}}, \F} \colon \pi^{-1}(M_{\mathrm{reg}, \F}) \to M_{\mathrm{reg}, \F}$ is one-to-one; the pullback $\presuper{\pi}{\F}$ exists and $\presuper{\pi}{\F}_{|\pi^{-1}(M_{\mathrm{reg}, \F})} \cong \F_{M_{\mathrm{reg}, \F}}$. 

These examples call for a more integrated approach to the holonomy groupoid and its related geometry.
To this end, we introduce the category of virtual manifolds, where an object of this category is a sheaf that reconstructs an atlas which may desingularize a foliation in a particular way; the global object of the sheaf is a groupoid internal to the site structure. These groupoids are not universal integrating groupoids, but may become so, e.g. the $s$-connected component of such a groupoid results in a universal integrating groupoid for the case of a foliation that is defined by an almost injective Lie algebroid over the ambient category of smooth manifolds. In this sense, the atlas described by a given sheaf or virtual manifold prescribes also the particular singularities that are encoded by the original foliation we started from. Another motivation for this study was the idea that it is useful to localize the resulting bicategory of sheaves, which synthesizes the structure of a general atlas, at the regular parts, $M_{\mathrm{reg}}$.

\subsection{The ambient site structure}

\label{subsection:axioms}

The axioms of the ambient category concern a structure encoded by the $3$-tuple $(\Aa, \Tau, \Ee)$, where $\Aa$ is the ambient category, $\Tau$ is a Grothendieck site and $\Ee$ is a suitable collection of admissible epimorphisms. We make the following assumptions:

\begin{enumerate}
%\item The ambient site $(\Aa, \Tau)$ admits a version of effective descent \ref{Def:effectivedescent}.
\item The triple $(\Aa, \Tau, \Ee)$ forms a site-structure \ref{Def:admepis}. \label{item:site-structure}
\item The site $\Tau$ is subcanonical \ref{Def:subcanonical} and its covers are local \ref{Def:local}. \label{item:subcanonical-local}
\item Any $\Ee$-sheaf over a \v{C}ech groupoid of an $\Ee$-cover, together with some bundle projection, is part of a principal bundle. In addition, this property is $\Tau$-local. \label{item:principal-bundle}
\item There is a fully faithful embedding functor $\Phi$ of $\Aa$ into the category of diffeological spaces, $\Diff$, with the compatibility properties listed below. \label{item:embedding-functor}
\end{enumerate}

The first three of these conditions suffice to define groupoids internal to $\Aa$ in the style of Meyer-Zhu \cite{mz}, as well as study their generalized tensor products. The fourth condition is completely specified in the following subsection below; it is used to construct the generalized sheaves studied in this work and to show that they reconstruct internal groupoids as desired.

\subsection{Embedding into the category of diffeological spaces}

The ambient site structure $(\Aa, \Tau, \Ee)$ is assumed to embed into the category of diffeological spaces $\Diff$. The latter category is also endowed with a suitable Grothendieck topology $\TauDopen$ which is comprised of covers that are $D$-open sets. There is also a class of admissible epimorphisms in $\Diff$, the local subductions, which are denoted by $\Eeloc$.The second set of assumptions concerns the existence and properties of a fully faithful functor
\[ 
\Phi \colon \Aa \to \Diff.
\]
This embedding functor is assumed to be continuous, by which is meant that it maps coverings to coverings, preserves the pullbacks of coverings and furthermore, $\Phi$ is assumed to preserve finite products (the ones that exist in $\Aa$). % as well as is compatible with the diadic structures. 
There is a canonical functor $D \colon \Diff \to \Top$ which equips a diffeological space with the $D$-topology; the largest topology such that all plots are continuous. We point out that $D$ has not as many desirable properties. For example, while it preserves colimits, it fails to preserve products in general and it is not full. The usual solution for the lack of product-preservation is to co-restrict $D$ to $\Delta$-generated spaces (e.g. \cite{csw}), i.e. we consider the functor:
\[
D^{\Delta} \colon \Diff \to \Top^{\Delta}. 
\]

Altogether, we assume our ambient category to fit into the functorial extension
\[
\xymatrix{
\Aa \ar[r]^{\Phi} & \Diff \ar[r]^{D^{\Delta}} & \Top^{\Delta}.
}
\]

The site structures on $\Aa$ and $\Diff$ combined with their compatibility allow us to talk about and study sheaves on these categories. Given the assumptions, the category of sheaves on $\Aa$ compares well with the category of sheaves on $\Diff$. In the case of the category of finite dimensional smooth manifolds, $\Aa = \Man$, the pullback functor furnishes an equivalence of these categories by the Grothendieck-Verdier theorem, cf. \cite{sw}. The same cannot be said for $\Diff$ in comparison with $\Top^{\Delta}$. %Therein lies the difficulty in attempting to generalize constructions of fibered objects in the ambient category that involve glueing of topological data. 
%We study the problem to describe a sheaf-structure over $\Aa$ that recovers certain instances of the holonomy groupoids commonly studied in the theory smooth manifolds. The data involves spans of local Morita equivalences and covers which consist of spans of admissible epimorphisms, which glue - under suitable conditions on the descent data - to an up to isomorphism unique internal groupoid. 

\section{$\Tau$-local internal Groupoids and Spanoids}

Let $(\Aa, \Tau)$ denote the ambient site and $\Ee = \Ee(\Tau)$ the maximal collection of $\Tau$-universal epimorphisms that are stable with respect to pullback, cf. \cite{roberts}. 

\begin{Def}
A \emph{site} $\Tau$ on an ambient category $\Aa$ is a collection of families $\{(U_i \to X)_{i \in I}\}_{X \in \Aa_0}$ of arrows in $\Aa_1$ called \emph{covers} between objects of $\Aa_0$ called \emph{opens}. A site verifies the following conditions: 
\begin{enumerate}
\item We have that $\{U' \iso U\} \in \Tau$, i.e. all \emph{isomorphisms} are contained in $\Tau$.

\item Given $\{U_i \to U\}_{i \in I}$ collection of covers in $\Tau$, then for each $V \to U \in \Aa_1$ we have that $\{U_i \ast_U V \to V\}_{i \in I}$ is in $\Tau$, i.e. $\Tau$ is closed with regard to \emph{base change}. 

\item Given $\{U_i \to U\}_{i \in I}$ in $\Tau$ and $\{U_{ij} \to U_i\}_{j \in J}$ in $\Tau$ for each $i \in I$, we have that the composite $\{U_{ij} \to U\}_{(i,j) \in I \times J}$ is also in $\Tau$, i.e. $\Tau$ is closed with regard to \emph{intersections}.
\end{enumerate}

If each cover consists of a single map, the site is called \emph{singleton site}. 

\label{Def:site}
\end{Def}

\begin{Def}
A site $(\Aa, \Tau)$ fulfilling the conditions of Proposition \ref{Prop:subcanonical} is called \emph{subcanonical}. 
\label{Def:subcanonical}
\end{Def}

% subcanonical site structure
% principal bundles

\begin{Prop}[\cite{mz}, Lem. 2.2]
Let $(\Aa, \Tau)$ be a site. The following conditions are equivalent

\begin{enumerate}
\item Each cover $f \colon U \to X$ in $\Tau$ is a coequalizer of some parallel arrows $g_1, g_2 \colon Z \rightrightarrows U$. 

\item Each cover $f \colon U\to X$ in $\Tau$ is the coequalizer of $\pr_1, \pr_2 \colon U \ast_f U \rightrightarrows U$. 

\item For each $W \in \Aa_0$ and $f \colon U \to X$ in $\Tau$, we have a bijection
\begin{align*}
\Mor_{\Aa}(X, W) &\iso \{h \in \Mor_{\Aa}(U, W) : h \circ \pr_1 = h \circ \pr_2 \ \text{in} \ \Mor_{\Aa}(U \ast_f U, W)\}, \\ 
& g \mapsto g \circ f. 
\end{align*}

\item All representable functors $\Mor_{\Aa}(-, W)$ on $\Aa$ are sheaves. 
\end{enumerate}

\label{Prop:subcanonical}
\end{Prop}

\begin{Def}
Let $(\Aa, \Tau)$ be a site. We say that $f \colon Y \to X \in \Aa_1$ has property $(P)$ \emph{$\Tau$-locally} if there is a cover $g \colon U \to X$ such that $\pr_2 \colon Y \ast_{f}^{g} U \to U$ has property $(P)$. 
\label{Def:local}
\end{Def}

% (3) in particular implies that \Tau is a sieve
% an element U \to X is called a cover
% we can always consider the single map (f_i)_{i \in I} \colon \bigsqcup_{i \in I} U_i \to X in lieu of the 
% collection of covers \{U_i \to U\}_{i \in I}

The condition (3) in particular implies that $\Tau$ is a sieve-structure, reflecting the alternative formulation of the three axioms to be found in the literature, in terms of the notion of a covering sieve. We can always consider a single map $(f_i)_{i \in I} \colon \bigsqcup_{i \in I} U_i \to X$ in lieu of the collection of covers $\{U_i \to U\}_{i \in I}$. 

\begin{Def}
The site $\Ee = \Ee(\Tau)$ is the singleton pre-topological site of $\Tau$-universal epimorphisms, i.e. arrows that are locally sectionable in the sense that $f \in \Ee$ is such that there is a covering family $\{c_i \colon U_i \to X\}$ and lifts $p_i \colon U_i \to S$ such that $f \circ p_i = c_i$. In addition, $f$ is assumed universal with this property. We call the triple $(\Aa, \Tau, \Ee)$ a \emph{site structure}.
\label{Def:admepis}
\end{Def}

\begin{Rem}
One can check that the subcanonicity assumption on $\Tau$ implies that $\Ee = \Ee(\Tau)$ is subcanonical. 
\label{Rem:admepis}
\end{Rem}

In what follows we make use of elementwise notation for groupoids and actions, which can be justified by the algorithm described in \cite[Section 3]{mz}. Denote by $\AS(\Ee)$ the category of $\Aa$-internal spans over $\Ee$ with internal span arrows as the arrows of the category. We introduce next $\Tau$-local internal groupoids, by which is meant a span that has a multiplication which furnishes a groupoid in a $\Tau$-local sense. 

\begin{Def}
A $\Tau$-local internal groupoid in $(\Aa, \Tau, \Ee)$ consists of the data $(\L, s, r) \in \AS(\Ee)_0$, $u \in \Mor_{\Aa}(\L_0, \L_1)$ and $\Aa$-isomorphism $i \colon \L_1 \iso \L_0$ so that $u$ and $i$ are compatible with the span-structure:
\begin{enumerate}
  \item 
  \[
    s \circ i = r, \quad r \circ i = s.
  \]

  \item 
  \[
    u^{\ast} s_{|\L_0} = u^{\ast} r_{|\L_0} = \id_{\L_0}.
  \]
  
  Multiplication is an arrow 
  \[
    p \colon \L_2 = \L_{1} \, {_{s}\ast_{r}} \, \L_1 \to \L_1
  \]
  such that there is a cover 
  \[
    c \colon D^2\L \to \L_2
  \]
  in $\Tau$ with 
  \[
    \pr_2 \colon \L_2 \ast_{c} D^2 \L \to D^2 \L
  \]
  being compatible with the remaining groupoid structure $(\L_0, \L_1, r, s, u, i)$: 
  
  \item 
  \[
    s \big( (p \circ \pr_2)(\gamma_1, \gamma_2) \big) = s(\gamma_2),
  \]
  
  \[
  	r \big((p \circ \pr_2)(\gamma_1, \gamma_2)\big) = r(\gamma_1), \quad (\gamma_1, \gamma_2) \in \L_2 \ast_c D^2 \L.
  \]
  
  \item 
  \[
  	(i \circ (p \circ \pr_2))(\gamma_1, \gamma_2) = (p \circ \pr_2)(i(\gamma_1), i(\gamma_2)), \quad 		(\gamma_1, \gamma_2) \in \L_2 \ast_c D^2 \L.
  \]
  
  \item 
  \[
  	(p \circ \pr_2)((u \circ r)(\gamma), \gamma) = \gamma = (p \circ \pr_2)(\gamma, (u \circ s)				(\gamma)).
  \]
  
  \item 
	\begin{align*}
    	&(p \circ \pr_2)\big( (p \circ \pr_2)(\gamma_1, \gamma_2), \gamma_3 \big) = (p \circ \pr_2)\big( 		\gamma_1, (p \circ \pr_2)(\gamma_2, \gamma_3) \big), \\
    	&\quad (\gamma_1, \gamma_2) \in \L_2 \ast_c D^2 \L, \\
    	&\quad (\gamma_2, \gamma_3) \in \L_2 \ast_c D^2 \L.
  	\end{align*}
\end{enumerate}

\label{Def:locgrpd}
\end{Def}

\begin{Def}
An internal groupoid $\G$ is called a \emph{spanoid} if for any span $S$ over $\Ee$ there is at most one span arrow $S \to \G$. Analogously, a $\Tau$-local spanoid is a $\Tau$-local groupoid with this property.
\label{Def:spanoid}
\end{Def}

We recall some of the constructions from \cite{debord} and how they carry over to our setting; see also \cite{bigonnet}.

\begin{Rem}
\emph{1)} Let $\L \rightrightarrows \L_0$ be a $\Tau$-local spanoid and $D \to \L_2$ a cover, then there is at most one arrow $p \colon D \to \L$ such that $s \circ p = s \circ \pr_2, \ r \circ p = r \circ \pr_1$. Therefore, a maximal cover $D_{\mathrm{max}}^2 \L \to L_2$ exists. It is understood from now on that we always fix the maximal cover for any given $\Tau$-local spanoid.

\emph{2)} Let $\varphi \colon \L \to \widetilde{\L}$ be a span-arrow for a $\Tau$-local groupoid $\L$ and a $\Tau$-local spanoid $\widetilde{\L}$. Then $\varphi$ is promoted to a strict groupoid morphism. The reasoning is that the spanoid property fixes the inverse and multiplication as uniquely defined internal arrows.
\label{Rem:locgrpd}
\end{Rem}

% defn site structure: admissible epi's & site
% subcanonical sites
% site: isomorphism, base change, intersections
% \Tau-local internal bibundle correspondences
% bicategory with \Tau-local bibundle correspondences & Morita isomorphisms

\begin{Prop}
A ($\Tau$-local) groupoid is a ($\Tau$-local) spanoid if and only the only local section of both the source and the range arrow is the unit arrow.
\label{Prop:spanoid}
\end{Prop}

\begin{proof}
Let $\L$ be a ($\Tau$-local) spanoid and let $\sigma$ be a local section of $r$ and $s$ for a cover $c_U \colon U \to \L_0$, i.e. $r \circ \sigma = c_U, \ s \circ \sigma = c_U$. Then $c_U^{\ast} u$ is a span arrow between the spans $\L_0 \xleftarrow{c_U} U \xrightarrow{c_U} \L_0$ and $\L_0 \xleftarrow{r} \L \xrightarrow{s} \L_0$. By the $\Tau$-universality of the epimorphisms $\Ee$ in $\Tau$, we can take $\bigsqcup_i U_i$ for a covering family $\{U_i \to \L_0\}_{i \in I}$ sufficient to cover $\L_0$, so that $c_U$ is in $\Ee$. Since $\L$ is a spanoid by assumption, we have that there is at most one span-arrow $U \to \L$ and hence $\sigma = c_U^{\ast} u$. For the other direction, let $\L$ have the property that the only $\Tau$-local section of $r, s$ is the unit arrow. Let $S$ be a span with arrows $a, b$ to $\L_0$ such that $f_1, f_2 \colon S \rightrightarrows S$ are two span arrows with $r \circ f_i = a, \ s \circ f_i = b$. 
Choose a $\Tau$-local section $\sigma \colon U \to S$ of $a$ and define $\nu \colon U \to \L$ via $x \mapsto (f_1(\sigma(x))^{-1} \cdot f_2(\sigma(x))$. 
\[
\begin{tikzcd}[column sep=3cm, row sep=2cm]
& \L \ar[d, shift left, "r"] \ar[d, swap, "s"] & \ar{l}{c_U^{\ast} u} U \ar[d, shift left, "c_U"] \ar[d, swap, "c_U"] \\
U \ar{ur}{\sigma} \ar{r}{c_U} & \L_0 \ar[u, bend left, "u"] \ar{r} & \L_0 
\end{tikzcd}
\]

Apply this process to a singleton cover $U = \bigsqcup_{i} U_i \to S$ in $\Ee$, so that 
\[
r \circ \nu = (a \circ \sigma)(x) = c_U(x), \ x \in U. 
\]

Hence, $\nu$ is an $\Ee$-section of both $r, s$. Therefore $f_1 \circ \sigma = f_2 \circ \sigma$. Since this works for any singletonized $\Tau$-local sections of $a$, it follows that $f_1 = f_2$. 
\end{proof}

We record the following strenghening which is relevant for an alternate approach to atlas constructions via sheaves of germs, as will be discussed below.

\begin{Prop}
Let $\L \rightrightarrows \L_0$ be a $\Tau$-local groupoid. The following assertions are equivalent:

\emph{1)} $\L$ is isomorphic as a span to a $\Tau$-local groupoid $\widetilde{\L} \rightrightarrows \L_0$ such that the only $\Tau$-local section of $\widetilde{r}, \widetilde{s}$ is the inclusion of the unit arrow. 

\emph{2)} $\L$ is isomorphic as a span to a $\Tau$-local spanoid. 
\label{Prop:spanoid2}
\end{Prop}

In order to prove this statement, we introduce notation that deals with restrictions of $\Tau$-local groupoids. To this end, let $c_V \colon V \to \L$ be a cover and $V^{+}$ all arrows $\gamma \in V$ such that $\gamma^{-1}$ is in $V$. Set $\L(V) := \L {_{s} \ast_{c}} V {_{c} \ast_{r}} \L \rightrightarrows V$, where we denote the range and source arrows by $\widetilde{r}$ and $\widetilde{s}$ respectively. Assuming $V {_{\id} \ast_u} \L_0$ exists, then with $\L(V \ast \L_0) = \L {_{s} \ast_u} (V \ast \L_0) {_{u} \ast_r} \L$ set $\L^{+}(V) := V^{+} {_{s} \ast_{\widetilde{r}}} \L(V \ast \L_0) {_{\widetilde{s}} \ast_r} V^{+} \rightrightarrows V \ast \L_0 =: \L^{+}(V)_0$. This furnishes a $\Tau$-local groupoid with domain:
\[
D^2 \L^{+}(V) := D^2 \L \ {_{p} \ast} \ (\L^{+}(V) {_{s} \ast_s} \L^{+}(V)). 
\]

Let two $\Tau$-local groupoids $\L \rightrightarrows \L_0, \ \widetilde{\L} \rightrightarrows \widetilde{\L}_0$ and let $\L_0 \ast \widetilde{L}_0 \in \Aa_0$ be given. Then $\L$ is isomorphic to $\widetilde{\L}$ as spans if there are covers $c \colon V \to \L_0 \ast \widetilde{\L}_0$ and $\widetilde{c} \colon \widetilde{V} \to \L_0 \ast \widetilde{\L}_0$ as well as an internal isomorphism $\varphi \colon V \to \widetilde{V}$ such that $\widetilde{s} \circ \varphi = s, \ \widetilde{r} \circ \varphi = r$. 

\begin{proof}
By virtue of Proposition \ref{Prop:spanoid} we obtain the direction \emph{2)} $\Rightarrow$ \emph{1)}. Let $\varphi \colon \widetilde{V} \to V$ be an isomorphism between two spans, where $\widetilde{V} \to \L_0 \ast \widetilde{\L}_0, \ V \to \L_0 \ast \widetilde{\L}_0$ such that $\widetilde{s} \circ \varphi = s, \ \widetilde{r} \circ \varphi = r$. We note that the domain $D^2 \L$ covers $\{(\gamma, \gamma^{-1}) : \gamma \in \L\}$ in the sense that there is a cover $c_W \colon W \to \L$ with the following properties: 
\begin{itemize}
\item $W = W^{-1}$, 

\item $c \colon (W \times W) \ast \L_2 \to D^2 \L$ is a cover,

\item $c^{\ast} p_{\L} \colon (W \times W) \ast \L_2 \to V$ is a cover.
\end{itemize}

Also, pulling back $r, s$ via $c_W$ induces a $\Tau$-local groupoid structure on $W$:
\[
\begin{tikzcd}[column sep=3cm, row sep=2cm]
W \ar[d, shift left, "c_W^{\ast} r"] \ar[d, swap, "c_W^{\ast} s"] \ar{r}{c_W} & \L_1 \ar[d, shift left, "r"] \ar[d, swap, "s"] \\
\L_0 \ar[equal]{r} & \L_0 
\end{tikzcd}
\]

By construction, $W$ and $\L$ are $\Tau$-locally isomorphic as spans. Let $\L_0 \xleftarrow{a} S \xrightarrow{b} \L_0$ be a span with two span arrows $f_1, f_2 \colon S \to W$. Given a singleton cover $c := \bigsqcup_i c_i \colon U := \bigsqcup_i U_i \to \L_0$ and local sections $\sigma := \bigsqcup_i \sigma_i \colon U \to S$, $c \in \Ee$, define 
\[
\nu \colon U \to W, \ x \mapsto f_1(\sigma(x))^{-1} f_2(x).
\] 
Then $\nu$ is an $\Ee$-local section for $r$ and $s$, thereby it is the unit arrow and $f_1 \circ \sigma = f_2 \circ \sigma$. Repeating this line of argument for any such constructed $\Ee$-local section of $a$, we obtain that $f_1 = f_2$. 
\end{proof}

\begin{Def}
Let $\L \rightrightarrows \L_0$ denote a $\Tau$-local groupoid with domain $D^2 \L$. A \emph{right $\Tau$-local action} $Z \ \rotatebox[origin=c]{90}{$\circlearrowright$}\ \L$ is implemented by $\alpha$ where $c_{\alpha} \colon D_{\alpha} \to Z {_{q} \ast_r} \L$ is a cover, and its \emph{anchor} is denoted by $q \colon Z \to \L_0$, such that the following conditions hold:
\begin{enumerate}
  \item 
  \[
    (q \circ \alpha)(z, \gamma) = s(\gamma), \ (z, \gamma) \in D_{\alpha}.
  \]

  \item 
  Given that $(z, \gamma_1) \in D_{\alpha}, \ (\gamma_1, \gamma_2) \in D^2\L$ and if one of $\alpha(\alpha(z, \gamma_1), \gamma_2)$ or $\alpha(z, m(\gamma_1, \gamma_2))$ is defined, then so is the other and they are equal.

  \item   
  $(z, \id_{q(z)})$ is contained in $D_{\alpha}$ and $\alpha(z, \id_{q(z)}) = z$, for all $z \in Z$. 
\end{enumerate}

The action is called an \emph{$\Ee$-sheaf} if $q \in \Ee$. The action is abbreviated as $\alpha(z, \gamma) = z \cdot \gamma$. Denote by $\Upsilon$ the arrow $(\pr_1, \alpha) \colon D_{\alpha} \to Z {_{q} \ast_q} Z$. The action is called \emph{principal} if there is a covering family $\{c_i \colon V_i \to Z\}_{i \in I}$ so that $\widetilde{c_i} \colon V_i \ {_{q} \ast_r} \ \L \to D_{\alpha}$ is a cover, induced by $c_i$, with the property that the arrow 
\[
\widetilde{c_i}^{\ast} \Upsilon \colon D_{\alpha} \ {_{\Upsilon} \ast_{\widetilde{c_i}}} \ (V_i \ {_{q} \ast_r} \ V_i) \iso V_i \ {_{q} \ast_q} \ V_i
\]

furnishes an isomorphism for each $i \in I$. 

\label{locaction}
\end{Def}

The $\Ee$-sheaves and principal actions are also being studied in \cite{mz, wockel} in order to construct the bicategory of internal groupoids, $\AGbi$. The aim here is to extend these constructions in order to compose $\Tau$-local bibundle correspondences. Let us fix some useful notation.

\begin{Not} 
We call $\Sh(\Ee)_{\L}$ the category of $\Ee$-sheaves, which consists of right actions by $\L$ where the anchor of the action is contained in $\Ee$ and the arrows are given by the obvious equivariant maps. Dually, $_{\L} \Sh(\Ee)$ denotes the corresponding category of left actions that are sheaves. By $\Aa_{\L}$ we denote the category of right actions and dually, by $_{\L} \Aa$ the category of left actions. Given an admissible epimorphism $p \colon X \to Z \in \Ee$, denote by $\Cech(p)$ the so-called \v{C}ech groupoid, given by the data 
\[
\Cech(p)_0 = X, \ \Cech(p)_1 = X {_{p} \ast_p} X, \ \check{r}(x_1, x_2) = x_1, \ \check{s}(x_1, x_2) = x_2 \ \text{for} \ p(x_1) = p(x_2); 
\]
the inverse $\check{\iota}(x_1, x_2) = (x_2, x_1)$ and multiplication $(x_1, x_2) \cdot (x_2, x_3) = (x_1, x_3)$ for $p(x_1) = p(x_2) = p(x_3)$. 
\label{Not:Esheaves}
\end{Not}

\begin{Rem}
We recall the definition of orbit spaces \cite[Lem. 5.3]{mz}. Let $Z \ \rotatebox[origin=c]{90}{$\circlearrowright$}\ \G$ be a given right action of a ($\Tau$-local) internal groupoid on $Z$, then the orbit-space is the coequalizer:
\[
\begin{tikzcd}
Z {_q} \ast_r \mathcal{G} \ar[r, shift left, "\pr_1"] \ar[r, swap, shift right, "\alpha"] & Z \ar{r}{\pi_Z} & Z / \G.
\end{tikzcd}
\]

If the action is a principal $\G$ bundle, then $q$ is equivalent to $\pi_Z$ and $\pi_Z \in \Ee$.
\label{Rem:orbitspace}
\end{Rem}

% $\Tau$-local bibundle correspondence
\begin{Def}
\emph{a)} A \emph{$\Tau$-local bibundle equivalence or Morita equivalence} between $\Tau$-local groupoids $\L(1)$ and $\L(0)$ is given by the following data and relations:
\[
\begin{tikzcd}[every label/.append style={swap}]
\L(1) \ar[d, shift left] \ar[d] \ar[symbol=\circlearrowleft]{r} & \ar{dl}{p_f} Z_f \ar{dr}{q_f} & \ar[symbol=\circlearrowright]{l} \L(0) \ar[d, shift left] \ar[d] \\
\O_1 & & \O_0
\end{tikzcd}
\]

\emph{Notation:} $f \colon \Bun{\L(1)}{Z_f}{\L(0)}$.

\begin{enumerate}
	\item The actions are commuting $\Tau$-principal actions.
	
	\item Denoting by $\{(V_i \to Z_f\}_{i \in I}$ the defining covering family, $q_f$ induces an isomorphism $V_i / \L(0) \iso q_f(V_i)$ and $p_f$ induces an isomorphism $V_i / \L(1) \iso p_f(V_i)$ for each $i \in I$. 
	
	\item $(p_f, q_f)$ furnish an arrow $Z_f \to \O_1 \times \O_0$ contained in $\Ee$.
\end{enumerate}

\emph{b)} We denote the equivalence class of $\Tau$-local Morita equivalences by $[\Bun{\L(1)}{Z_f}{\L(0)}] =: f \colon \L(1) \dashedrightarrow \L(0)$ and we call this a \emph{Morita isomorphism}, where the equivalence relation is defined as the span-isomorphisms that intertwine the actions. A representative of the class is called a \emph{span for $f$}. 

\emph{c)} A \emph{bi-bundle correspondence} between $\L(1)$ and $\L(0)$ is given by the same data, but with the variation on the relations that the right action is principal, $q_f \in \Ee$ and $p_f$ induces the isomorphisms $V_i / \L(1) \iso p_f(V_i), \ i \in I$. 
\label{Def:locbibundle}
\end{Def}

\begin{Lem}
A $\Tau$-local Morita isomorphism $f$ between $\Tau$-local spanoids is entirely determined by a span for $f$, up to span isomorphism.
\label{Lem:locbibundle}
\end{Lem}

\begin{proof}
Let $V_i \to Z_f$ be a cover that is part of a covering family with the property that $V_i  / \L(0) \iso p_f(V_i)$ is an isomorphism induced by $p_f$. Given a span $\O_1 \xleftarrow{b} S \xrightarrow{a} \O_0$ and arrows $g_0, g_1 \colon S \to V_i$ such that $q_f \circ g_0 = a = q_f \circ g_1$ and $p_f \circ g_0 = b = q_f \circ g_1$. Define $h \colon S \to \L(0), \ x \mapsto \gamma$ via $g_0(x) = g_1(x) \cdot \gamma$ and $\widetilde{h} \colon S \to \L(0)$ via $x \mapsto a(x)$. Then $s_0 \circ h = s_0 \circ \widetilde{h} = a$ and $r_0 \circ h = r_0 \circ \widetilde{h} = a$. If $\L(0)$ is a $\Tau$-local spanoid, there is at most one span arrow between $\O_0 \xleftarrow{a} S \xrightarrow{b} \O_0$ and $\O_0 \xleftarrow{r_0} \L(0) \xrightarrow{r_0} \O_0$, hence $h = \widetilde{h}$ and $g_0 = g_1$. The assertion follows.
\end{proof}

The main result of this section is the composition theorem for $\Tau$-local bibundle correspondences. We make use of the assumptions (1) through (3) as stated in Section \ref{subsection:axioms}. 

\begin{Thm}
Let $g \colon \Bun{\L(2)}{Z_g}{\L(1)}$ and $f \colon \Bun{\L(1)}{Z_f}{\L(0)}$ be $\Tau$-local bibundle correspondences. There is a $\Tau$-local subgroupoid cover $\H_1 \to \L(1)$, as well as covers $c_g \colon V_1 \to Z_g, \ c_f \colon V_2 \to Z_f$ so that 
\[
g_{|\L(2), V_1, \H_1} \colon \Bun{\L(2)}{V_1}{\H_1}, \ f_{|\H_1, V_1, \L(0)} \colon \Bun{\H_1}{V_2}{\L(0)} 
\]

become composable in the sense that the product
\[
V_1 \circledast_{\H_1} V_2 := V_1 {_{c_g^{\ast} q_g} \ast_{c_f^{\ast} p_f}} V_2 / \H_1
\]

exists as an object in $\Aa_0$ and carries the canonical actions $\Bun{\L(2)}{V_1 \circledast_{\H_1} V_2}{\L(0)}$ that furnish a $\Tau$-local bibundle correspondence. 

\emph{Notation:} We denote the product by $Z_f \circledast_{\mathrm{loc}} Z_g$, with underlying span $Z_g \ast_{\mathrm{loc}} Z_f$. 
\label{Thm:locbibundle}
\end{Thm}

\begin{proof}
For a given admissible epimorphism $p \colon X \to Z \in \Ee$, consider $Z$ as a $0$-groupoid, $X$ an equivalence, then we have the induced functor
\[
\Sh(\Ee)_Z \to \Sh(\Ee)_{\Cech(p)}.
\]

The application of this functor, combined with a straightforward adapation of the argument of \cite[proof of Prop. 7.6]{mz}, making appropriate use of the slightly stronger assumption \ref{item:principal-bundle}, furnishes the functor
\[
\Sh(\Ee)_{\L(1)} \to _{\L(0)}\Aa, \ W \mapsto Z \circledast_{\mathrm{loc}} W.
\]

Here $\Bun{\L(0)}{Z_g}{\L(1)}$ denotes the given $\Tau$-local bibundle correspondence.
\[
\begin{tikzcd}[every label/.append style={swap}]
& V_1 \ar{d}{c_g} & \ar[symbol=\circlearrowright]{l} \H_1 \ar{d} \ar[symbol=\circlearrowleft]{r} & V_2 \ar{d}{c_f} & \\
\L(2) \ar[d, shift left] \ar[d] \ar[symbol=\circlearrowleft]{r} & \ar{dl}{p_g} Z_g \ar{dr}{q_g} & \ar[symbol=\circlearrowright]{l} \L(1) \ar[d, shift left] \ar[d] \ar[symbol=\circlearrowleft]{r} & \ar{dl}{p_f} Z_f \ar{dr}{q_f} & \ar[symbol=\circlearrowright]{l} \L(0) \ar[d, shift left] \ar[d] \\
\O_2 & & \O_1 && \O_0
\end{tikzcd}
\]

In other words, there exist subgroupoid covers $V_1, V_2$, that fit into the above diagram, in such a way that the local composition exists when constructed out of the induced coequalizer diagrams.
\end{proof}

\section{Virtual Manifolds and Groupoid Reconstruction}

% Pseudogroup of \Tau-local Morita isomorphisms
% Generalized pre-sheaves of \Tau-local Morita isomorphisms
% Virtual manifolds as abstract sheafification
% Thm: Equality of sheafification with the virtual manifolds (=sheaves) over diffeological spaces

% Rem: In the category $\Manec$, natural isomorphy with Aof-Brown sheaves.

% Recall other leading examples.

% $\circledast$
% $_{\L}\Sh(\Ee)$

The sheaves are constructed using a pseudo-group generated by $\Tau$-local Morita isomorphisms, see also \cite{debord}. We make use of the embedding functor into the category of diffeological spaces - using assumption \ref{item:embedding-functor} of Section \ref{subsection:axioms} - in order to specify a path holonomy diffeology on the quotient of the pseudo-group using the final diffeology generated by the natural projection. %The pre-sheaf is sheafified with an abstract construction, based on a structure of span-covers. 
The main result is that the sheafification produces a category of $\Ee$-sheaves which globalize to groupoids that are internal to the site structure. Conditions on the path holonomy diffeology correspond to certain properties of the underlying foliation, e.g. in the special case of the ambient category of smooth manifolds, projectivity and involutivity of a foliation are reflected by properties of the associated path-holonomy diffeology.  
%this is illustrated by way of examples that conclude this section.

\begin{Def}
A \emph{generalized atlas} $\U_X := \{\L_i \rightrightarrows U_i\}_{i \in I}$ consists of $\Tau$-local groupoids adapted to the covering family $\{\varphi_i \colon U_i \to X\}_{i \in I}$ such that
\begin{enumerate}
\item For each $i \in I$ the $\Tau$-local groupoid $\L_i$ is a $\Tau$-local spanoid.

\item For each $i, j \in I$ there are $\Tau$-local subgroupoids $\H_{i}^{j} \to \L_i, \ \H_{j}^{i} \to \L_j$, where the arrows are covers, and span-isomorphisms $\varphi_{ij} \colon \H_i^j \iso \H_j^i$.
\end{enumerate}

\label{Def:gendatlas}
\end{Def}

\begin{Rem}
The $\Tau$-local spanoid propery yields that $\varphi_{ij}$ are uniquely determined, i.e. there is a maximal cover $\mathcal{D}(\varphi_{ij}) \to \L_i$ such that it factors through $U_i \ast_U U_j \to \mathcal{D}(\varphi_{ij})$; called the \emph{domain} of $\varphi_{ij}$. 
\label{Rem:gendatlas}
\end{Rem}

We describe operations on a $\Tau$-local Morita isomorphism $f \colon \L(1) \dashedrightarrow \L(0)$ with a given span $\mathcal{O}_1 \xleftarrow{p_f} Z_f \xrightarrow{q_j} \mathcal{O}_0$ for $f$. 

\begin{itemize}
\item \emph{Identity:} $\id_{L(0)} \colon \L(0) \dashedrightarrow \L(0), \ Z_{\id_{\L(0)}} = \L(0)$, $p_{\id_{\L(0)}} = s(0), \ q_{|\id_{\L(0)}} = r(0)$, with actions given by right and left multiplication. 

\item \emph{Inversion:}  $Z_{f^{-1}} = Z_f, \ p_{f^{-1}} = q_f, \ q_{f^{-1}} = p_f$. The right $\Tau$-local action $Z_{f^{-1}} \ \rotatebox[origin=c]{90}{$\circlearrowright$}\ \L(1)$ (respectively $\L(0) \ \rotatebox[origin=c]{90}{$\circlearrowleft$}\ Z_{f^{-1}}$) is implemented by $\alpha_1(z, \gamma_1) = \gamma_1^{-1} \cdot z$ (respectively $\alpha_0(\gamma_0, z) = z \cdot \gamma_0^{-1}$). Then $f^{-1} \colon \L(0) \dashedrightarrow \L(1)$ is a $\Tau$-local Morita isomorphism with span $\mathcal{O}_0 \xleftarrow{q_{f^{-1}}} Z_{f^{-1}} \xrightarrow{p_{f^{-1}}} \mathcal{O}_1$ for $f^{-1}$. 

\item \emph{Restriction:} Let $\H_0 \to \L(0), \ \H_1 \to \L(1)$ be $\Tau$-local subgroupoid covers and a cover $V \to Z_f$ such that $p_f(V)$ is the units of $\H_0$ and $q_f(V)$ is the units of $\H_1$. The restriction of $f$, denoted by $f_{|\H_1, V, \H_0} \H_1 \dashedrightarrow \H_0$, admits as span the restriction $q_f(V) \xleftarrow{q_f} V \xrightarrow{p_f} p_f(V)$ with $\Tau$-local action induced by $\L(1)$ and $\L(0)$. 

\item \emph{Composition:} Explained by Theorem \ref{Thm:locbibundle}. 
\end{itemize}

\begin{Def}
Let $\U := \U_X$ be a given generalized atlas and assume that $\U$ is stable with regard to restriction. Then the pseudogroup $\Psi_{\U}$ is defined to consist of local Morita isomorphisms between elements of $\U$ such that the identity is contained in $\Psi_{\U}$ and $\Psi_{\U}$ is stable with regard to inversion, local composition and restriction.
\label{Def:pseudogroup}
\end{Def}

We introduce an equivalence relation on a pseudogroup and pass to the quotient; we do this, after utilization of the embedding $\Phi$, inside the category of diffeological spaces. Let $f \colon \L(1) \dashedrightarrow \L(0)$ and $g \colon \L(1) \dashedrightarrow \L(0)$ be $\Tau$-local Morita isomorphisms. Consider the corresponding spans $\mathcal{O}_1 \xleftarrow{p_f} Z_f \xrightarrow{q_f} \mathcal{O}_0, \ \mathcal{O}_1 \xleftarrow{p_g} Z_g \xrightarrow{q_g} \mathcal{O}_0$. By continuity, the functor $\Phi \colon (\Aa, \Tau, \Ee) \to (\Diff, \TauDopen, \Eeloc)$ maps the given spans to spans over $\Eeloc$ (local subductions). Introduce germs, where $[f]_{z_f} = [g]_{z_g}$ if there are $D$-open neighbhorhoods $V_{z_f}, \ V_{z_g}$ in $Z_f$ and $Z_g$ respectively and $\Eeloc$ span isomorphisms $\varphi \colon V_{z_f} \iso V_{z_g}$ such that $\varphi(z_f) = z_g$. Denote the underlying equivalence relation by $\sim$. Let us fix the notation $\PhiPsi_{\U} := \Phi_{\ast} \Psi_{\U}$ for the pseudo-group formed out of the spans over $\Eeloc$ that are images under the functor $\Phi$ of the corresponding spans over $\Ee$. Note that the span-generators of the pseudogroup $\Psi_{\U}$ form a cocycle. To this end, consider the span isomorphisms $\varphi_{ij} \colon \H_i^j \iso \H_{j}^i$, $\varphi_{jk} \colon \H_j^k \iso \H_k^j$ and $\varphi_{ik} \colon \H_i^k \iso \H_k^i$ between $\Tau$-local subgroupoids with given covers $\H_i^j \to \L(i), \ \H_i^k \to \L(i), \ \H_j^i \to \L(j), \ \H_j^k \to \L(j), \ \H_k^i \to \L(k), \ \H_k^i \to \L(k)$. Then there are span isomorphisms completing the diagram:
\[
\begin{tikzcd}[column sep=3cm, row sep=2cm]
\H_j^i \ar[d, dashed, ""] \ar[r, "\varphi_{ij}", "\simeq"'] & \H_i^j \ar[d, dashed, ""] \ar[r, dashed, ""] & \H_i^k \ar[d, "\varphi_{ik}", "\simeq"'] \\
\H_j^k \ar[r, "\varphi_{jk}", "\simeq"'] & \H_k^i \ar[r, dashed, ""] & \H_k^i
\end{tikzcd}
\]

Since the elements of $\PhiPsi_{\U}$ for a given $\U = \U_X$, are uniquely determined by their representing spans up to span-isomorphism, we can form the quotient $\PhiPsi_{\U} / \sim$ and endow it with the following so-called path holonomy diffeology $\P(\Psi_{\U})$. As preparation, let us recall some facts regarding diffeology, cf \cite{i}. 

\begin{Def}
The set $\P(X)$ consisting of maps $\chi \colon \O_{\chi} \to X$ with $\O_{\chi} \subset \Rr^n$ open for some $n \in \Nn$ is called a \emph{diffeology} on $X$ and its elements are called \emph{plots}, if the following conditions hold:

\begin{enumerate}
\item The constant functions $\Rr^n \to x \subseteq X$ are contained in $\P(X)$ for each $x \in X, \ n \in \Nn$.

\item If $\chi \colon \O_{\chi} \to X$ is such that for each $y \in \O_{\chi}$ there is an open subset $V \subseteq \O_{\chi}$ such that $\chi_{|V} \in \P(X)$, then $\chi \in \P(X)$. 

\item For each smooth $f \colon V \to \O_{\chi}$, $V \subseteq \Rr^n$ open, the composition with a given plot $\chi$, $\chi \circ f$ is contained in $\P(X)$. 

\end{enumerate}

\label{Def:diffeology}
\end{Def}

% quotient diffeology...
We endow the quotient $\PhiPsi_{\U} / \sim$ with the quotient diffeology $\P(\PhiPsi)$ which is the final diffeology induced by the projection $q \colon \PhiPsi_{\U} \to \PhiPsi_{\U} / \sim$.  The $D$-topology of this diffeology coincides with the quotient topology, when the pseudo-group is endowed with its canonical $D$-topology, cf. \cite{az, csw, i}. 

\begin{Prop}
Let $\U$ be some generalized atlas over $X \in \Aa_0$. 
\begin{enumerate}
\item The quotient map $q \colon \PhiPsi_{\U} \to \PhiPsi_{\U} / \sim$ is a local subduction.

\item Any two final diffeologies over (maximal) generalized atlases $\U$ and $\U'$ over $X$ have the same path-holonomy diffeology. 
\end{enumerate}

\label{Prop:pathholdiffeo}
\end{Prop}

\begin{proof}
We show \emph{(1)}, the proof of \emph{(2)} goes along similar lines; see also \cite{az} for a similar setting. Consider an element $\L(U) \rightrightarrows U$ in $\U_X$ and let $\H_U \to \L(U)$ denote the defining cocycle subgroupoid, up to a fixed span isomorphism using Lemma \ref{Lem:locbibundle}. Fix an element $w \in \H_U$, $q(w) = z$ and take a plot $\chi \colon \O_{\chi} \to \PhiPsi_{\U} / \sim$ and $x \in \O_{\chi}$ with $\chi(x) = z$. There is a connected open neighborhood $V$ of $x$ and a plot $\chi' \colon V \to \PhiPsi_{\U}$ so that $\chi_{|V} = q \circ \chi'$. Then $w' = \chi'(x)$ is in some $\H_{U'}$ and shrinking $U'$ if necessary, there is a span isomorphism $\varphi \colon \H_U \iso \H_{U'}$ with $q(w) = q(w') = z$ and $\varphi \circ \chi' \colon V \to \H_U$ furnishes a lifting.
\end{proof}

% Construct a diffeological fiber bundle $p \colon \Lambda_{\PhiM} \to X$ that is plotwise local.
% Basis of $D$-open sets induced by "germ-action" \dot{s}. 

Given $[f]_{z} \in \PhiPsi_{\U} / \sim$ and let $f \colon \L(1) \dashedrightarrow \L(0)$ be a representative with a given span $\O_1 \xleftarrow{p_f} Z_f \xrightarrow{q_f} \O_0$ and $z \in Z_f$. By Lemma \ref{Lem:locbibundle} the map $Z_f \to \PhiPsi_{\U} / \sim, \ w \mapsto [f]_w$ is $D$-locally injective. Also, define the map $X \to \PhiPsi_{\U} / \sim, \ x \mapsto [\id_{\L}]_x$ where $\L$ is some element of $\U$ and $x$ is contained in $\L_0$. 

% Algebraic groupoid structure of \G_{\U} := \PhiPsi_{\U} / \sim$. 
% Show: \G_{\U} is actually an $\Aa$-internal spanoid.

\begin{Thm}
The groupoid $\G := \PhiPsi_{\U} / \sim$ with units $X$ for a given generalized atlas $\U := \U_X$ with $X \in \Aa$ is a spanoid internal to the site structure $(\Aa, \Tau, \Ee)$. 
\label{Thm:internalization}
\end{Thm}

\begin{proof}
The unit inclusion $u \colon X \to \G$ which maps to the identities of the corresponding maximal spanoids. The source $s \colon \G \to X$ maps $[\Bun{\L(1)}{Z_f}{\L(2)}]_z \mapsto q_f(z)$ and $r \colon \G \to X$ maps $[\Bun{\L(1)}{Z_f}{\L(2)}]_z \mapsto p_f(z)$. The composition is defined by 
\[
[\Bun{\L(2)}{Z_g}{\L(1)}]_t \cdot [\Bun{\L(1)}{Z_f}{\L(0)}]_z = [\Bun{\L(2)}{Z_g \circledast_{\mathrm{loc}} Z_f}{\L(0)}]_{(t,z)}.
\]

The inverse is defined as
\[
i \colon \G \to \G, \ [\Bun{\L(1)}{Z_f}{\L(0)}]^{-1} = [\Bun{\L(0)}{Z_f^{-1}}{\L(1)}]
\]

and this furnishes the groupoid structure on $\G$. Let $U \to X$ be a cover in $\Tau$. Define $\nu \colon U \to \G$ to be a $\Tau$-local section of both $r$ and $s$. Take an element $\L \rightrightarrows U$ with source / range denoted by $\widetilde{s}, \ \widetilde{r}$ of $\U$. Denote by $f \colon \L \dashedrightarrow \L$ a $\Tau$-local Morita isomorphism implemented by the span $U \xleftarrow{p_f} \L \xrightarrow{q_f} U$ such that the image of $\nu$ is a subset of $\{[f]_z : z \in Z_f\}$. Then there is a unique arrow $\widetilde{\nu}$ in $\Aa$ such that $\nu(x) = [f]_{\widetilde{\nu}(x)}, \ x \in U$ and $p_f \circ \widetilde{\nu} = q_f \circ \widetilde{\nu} = \id_U$. Consider the arrow $\gamma \mapsto \gamma \cdot \widetilde{\nu}(\widetilde{s}(\gamma))$ denoted $\varphi \colon \L \to Z_f$. The multiplication sign here refers to the canonical $\Tau$-local action; in particular $\varphi$ induces an isomorphism between spans in $\Aa$. This isomorphism is mapped via $\Phi$ so that there is a $D$-local neighborhood $W$ of $U$ in $\L$ such that
\[
[\id_{\L}]_w = [f]_{\varphi(w)}, \ w \in W. 
\]

If $x \in U$, we get that $[\id_{\L}]_x = [f]_{\varphi(x)} = [f]_{\widetilde{\nu}(x)} = \nu(x)$. If $\G$ is $\Aa$-internal this will then furnish the spanoid property of $\G$ by Proposition \ref{Prop:spanoid}, since we just showed that $\nu$ is the inclusion of units. We are left to check that $\nu$ is an admissible monomorphism, or dually that $\nu^{\mathrm{op}} \colon \G \to X$ is an admissible epimorphism, i.e. a universally locally sectionable arrow. To this end, let us consider the decomposition $\G = \bigsqcup_{x \in X} \G_x$ into the sections and define the canonical action of $\G$ on itself with anchor given by the range $r$, as defined above. Let $(c_i \colon U_i \to X)_{i \in I}$ be the covering family that defines the generalized atlas $\U$. Then there is a lifting $f \colon U_i \to \G$ which is defined by taking a possibly larger $D$-open neighborhood $U$ and mapping it into a $\G_x$ via $\nu$, for a given $x \in U_i \subset U$ so that $r \circ f = c_i$. This property of $r$ is invariant under a change of base by Prop. \ref{Prop:pathholdiffeo} \emph{(2)}. Hence $r$ is an admissible epimorphism. 
\end{proof}

\end{document}